\documentclass[a4paper, 12pt]{article}

\usepackage[dvips]{graphics,color}
\def\spc{\color[rgb]{1,0.2,0.2}}
\def\nc{\normalcolor}
\usepackage[british]{babel}
\usepackage[pdftex]{graphicx}
\usepackage{graphics, setspace}

\usepackage{amsmath, amsthm, amsfonts, amsbsy, amssymb, scrextend, graphicx}
\usepackage{hyperref} 
\usepackage{verbatim}
\usepackage{color}
\usepackage{epstopdf}

\usepackage{color}

\usepackage{tikz}
\usepackage[export]{adjustbox}
\usepackage{float}

\newtheorem {corollary}{Corollary}

\newtheorem {theorem}{Theorem}
\newtheorem{remark}{Remark}

\newtheorem{example}{Example}

\newcommand{\R}{{\mathbb{R}}}

\renewcommand{\P}{\mathsf{P}}

\newcommand{\eps}{\varepsilon}
\newcommand{\E}{\mathsf{E}}

\newcommand{\K}{\mathcal{K}}
\newcommand{\bi}{{\bf 1}}

\renewcommand{\L}{\mathcal{L}}

\newcommand\Tau{\mathcal{T}}

\begin{document}

\date{}

\title{Discrete SIR model on a homogeneous tree and its continuous limit}

\author{
Alexander Gairat\footnote{
Numerix.  Email address: agairat@gearquant.com
}\, and 
Vadim Shcherbakov\footnote{
 Royal Holloway,  University of London.
 Email address: vadim.shcherbakov@rhul.ac.uk
}
}

\maketitle

\begin{abstract}
{\small

We study 
 a discrete Susceptible-Infected-Recovered (SIR) model for the spread of infectious disease on a homogeneous tree and  the limit behavior of the model in the case when the tree vertex degree 
tends to infinity. 
We obtain the distribution of the time it takes for a susceptible vertex to get infected 
in terms of a solution of a non-linear integral equation
under broad assumptions on the model parameters.
Namely, infection rates are assumed to be time-dependent, and recovery times are  given by random variables
with a fairly arbitrary distribution.
We then study the behavior of the model in the limit when the tree vertex degree
 tends to infinity, and  infection rates are appropriately scaled.
We show that in this limit the integral equation of the discrete model
implies an equation for the susceptible population compartment. 
This is a master equation in the sense that both the infectious and the recovered compartments 
can be explicitly expressed in terms of its solution.
}
\end{abstract}


\section{Introduction}
\label{intro}
 
Studying the spread of infectious disease has been of great interest for a long time and 
has motivated a lot of mathematical models.
Susceptible-Infected-Recovered (SIR) type  models are among those that are 
 commonly studied.
A SIR model is a  compartmental model, in which
 a population of individuals is divided into three distinct groups (compartments).
The first compartment consists of individuals that are  susceptible to the disease, but are not yet infected.
The second compartment is represented by infected individuals. 
Finally, the remaining third compartment is a group of individuals, who have been infected and recovered 
from the disease. 

In this paper we revisit a  discrete stochastic SIR model 
and study  its continuous limit (to be explained).
In the  discrete setting a population is modeled by vertices of a graph, and 
infection is transmitted from infected vertices to susceptible ones via edges of the graph.
Such a  model  is  usually  studied under additional  assumptions on
 the graph, infection rates and recovery times
(e.g., see~\cite{Andersson}, \cite{BrittonLN}, 
\cite{Fabricius}, \cite{Montagnon}, \cite{moreno}, \cite{Schutz}, \cite{zhang}, and 
references therein).
For example, if  the underlying graph is complete,
 infection rates are constant and the recovery  times are exponentially 
distributed,  then the model is a discrete version of
 the classic SIR model of A.G. McKendrick and W.O. Kermack (\cite{SIR-paper}). 

We study the discrete SIR model on a homogeneous tree.
  The latter is an infinite connected constant 
vertex degree graph without cycles. The constant vertex degree means that 
 each vertex has the same number of adjacent vertices (neighbors). 
A homogeneous tree can serve as a mathematical model for 
an infinite closed homogeneous population, in which  
all individuals have the same number of social contacts. 
To the best of our knowledge the SIR model on a homogeneous tree has never been 
considered, despite many years of study of the SIR model  on graphs (networks).

The discrete model is studied under broad assumptions 
on both infection rates and recovery times. 
In particular, we assume that  an infected vertex emits germs 
according to a Poisson process with a time-dependent rate.
A susceptible vertex can  be also  infected by itself (according to another Poisson process),
which can be interpreted as a source of infection outside of the population.
Infection rates are assumed to be  time-dependent deterministic functions.
An infected vertex recovers in a  period of time 
given by a random variable. Recovery times are assumed to be independent 
identically distributed random variables with a fairly arbitrary common distribution.

 Our main result for the discrete model is concerned with 
 the distribution of the time it takes for a susceptible 
vertex  to get infected (the time to infection).
We obtain a simple analytical expression for this distribution 
in terms of a  solution of a non-linear integral equation. 
In some special cases this integral equation is  equivalent  to a differential equation
of Bernoulli type (depending on a particular case). In one of these cases 
the  corresponding differential equation can be solved analytically. 

The structure of a homogeneous  tree plays an essential role 
in our analysis.  The key observation is that a susceptible vertex 
 splits  the homogeneous  tree into  a finite number of  
 identical subgraphs, and infection processes on these  
subgraphs are independent and identically distributed.

We then study the discrete model in the limit, as   
 the  vertex degree of the tree  tends to infinity, and the infection rate decreases
proportionally. We show that in this limit our results for  the discrete model 
imply an equation for the susceptible compartment. 
We call it the master equation, because both the  
infectious and the recovered compartments can be explicitly expressed 
in terms of its  solution. 
This  generalizes the results in
 (\cite{Harko}, \cite{Kendall} and~\cite{Kroger}), where 
the system of equations of  the classic  Kermack-McKendrick SIR model was reduced 
to a single equation. 

The obtained continuous SIR model is fairly general and provides a flexible 
technical framework for modeling various  infectious and recovery dynamics.
In fact, the master equation implies a family of continuous SIR models.
A particular SIR model depends on the structure and interpretation of the model parameters.
For example, the Kermack-McKendrick SIR model is 
a special case of our model. Another special case of the model 
coincides with the SIR model  proposed in~\cite{DellAnna}.

The rest of the paper  is organized as follows.
In Section~\ref{disc-model} we consider the discrete SIR model.
The model is formally defined in Section~\ref{disc-def}.
The main result for the discrete model 
is stated and proved 
in Section~\ref{disc-main}.
In Section~\ref{master-sec} we use this result
to derive the master 
equation for the susceptible compartment in the limit, 
 as the tree  vertex degree tends to infinity.
We discuss  the continuous SIR model implied by the master equation 
 in Section~\ref{special-continuous} and 
briefly comment of the relationship of our  model with fractional 
SIR models in Section~\ref{fractional}.
Finally, in  Section~\ref{examples-SIR} 
we consider some special cases in which 
the  integral equation of the discrete model can be written in the 
differential form.

\section{The discrete SIR model}
\label{disc-model}

\subsection{The model definition}
\label{disc-def}

We start with defining  the discrete  continuous time SIR  model on  a general graph (since a particular 
structure of the graph is not important at the definition).
Let  $\Tau$ be a connected (and possibly infinite) graph. 
With some abuse of notation,  we will associate a graph   with the set of its vertices. 
Given vertices  $x, y\in \Tau$ we write $x\sim y$, if these vertices   are connected 
by an edge, in which case we call them neighbors. 
Given a vertex its vertex degree is defined as  the number of its neighbors.
A vertex can be either susceptible, or infected, or recovered (and immune). 
At time $t=0$ each vertex is either susceptible, or 
infected (in which case we assume that it gets infected at time $t=0$).
When a vertex becomes infected, it starts emitting  infectious germs 
 towards all of its neighbors and continues to do so until the moment of its  recovery from the disease.
When a susceptible vertex gets a germ, it becomes infected, and all subsequently arrived germs do not give 
any additional effect. 
An infected vertex emits germs  to a given neighbor according to a 
Poisson process with a time dependent rate.
Namely,  if a  vertex $y\in \Tau$ becomes infected at time $t_y$ and recovers in a time given in the general case  by a random variable $H_y$,
then, at time $t\in [t_y, t_y+H_y]$ it 
infects a susceptible neighbor with the rate $\eps_{t-t_y}$,
where $(\eps_t,\, t\geq 0)$ is  a non-negative  deterministic  function.
A susceptible vertex  can be also   infected by itself, according to a Poisson 
process with the time-dependent  rate $\lambda_{t}$, where 
$(\lambda_{t},\, t\geq 0)$ is  a non-negative deterministic  function. 
Thus, given $t_y,\, H_y=h_y$ for $y\sim x$, a susceptible  vertex $x$ is infected at time $t$ 
with the following  total infection rate 
\begin{equation}
\label{r-ran0}
\lambda_{t} + \sum_{y: y\sim x}\eps_{t-t_y}\cdot
\bi_{\{t_y\leq t\leq t_y+h_y\}}.
\end{equation}
We  assume 
that all  random variables  are realized on a certain probability space $(\Omega, {\cal F}, \P)$,
 the expectation with respect to the probability $\P$ is denoted by $\E$;
all Poisson processes  are independent of each other, and they are also independent of
the recovery times.

\begin{figure}
\centering
\includegraphics[scale=0.32]{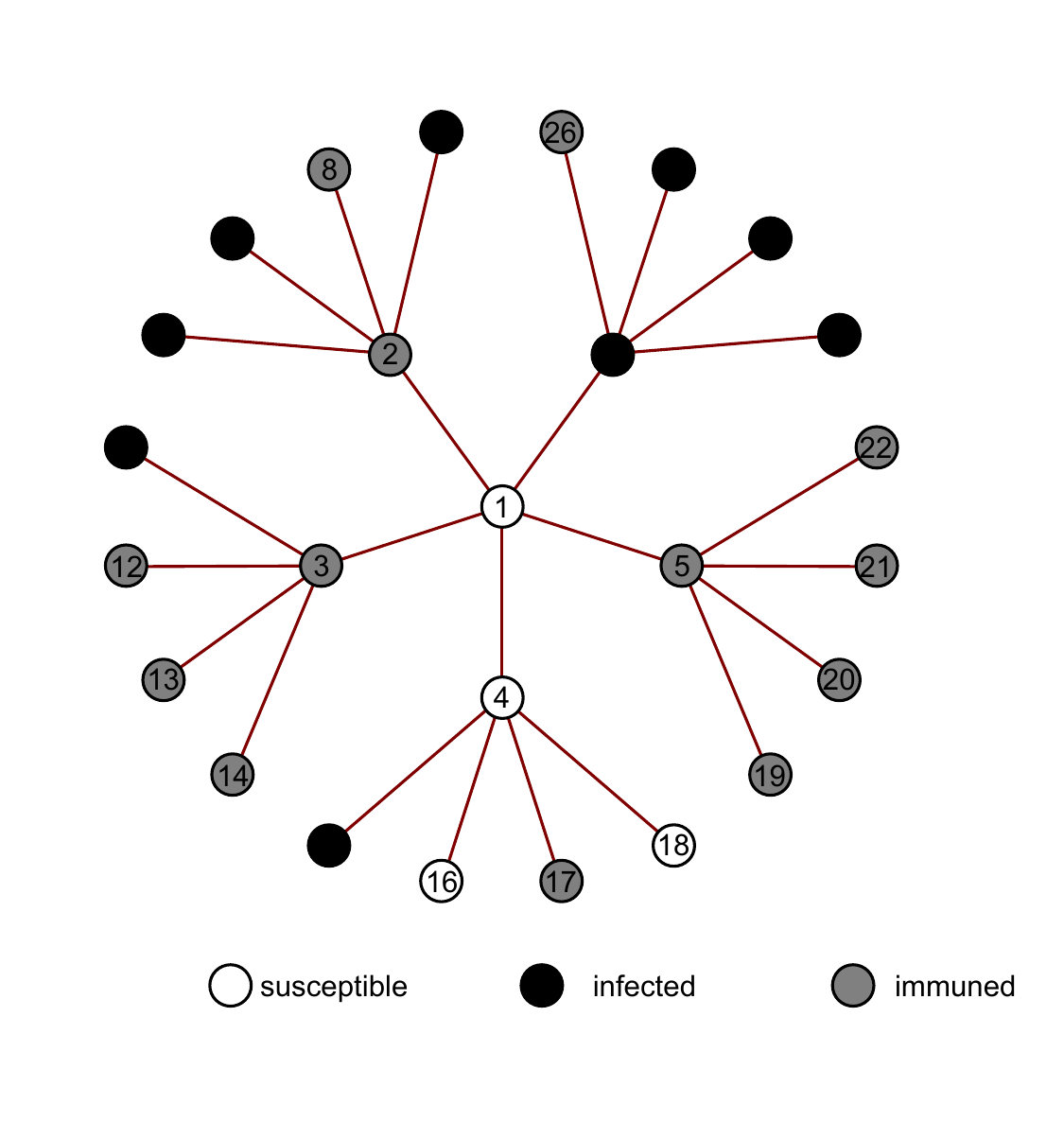}
\caption{\small A finite fragment of  a homogeneous tree, where $n=4$ (i.e. each vertex has $5$ neighbors).
}
\label{tree}
\end{figure}

Both infection rates and recovery times can be fairly arbitrary. However, 
some minimal technical assumptions are required.
These assumptions are as follows.
\begin{enumerate}
\item[(\bf A1)]  Functions $(\eps_t,\, t\geq 0)$ and  $(\lambda_t,\, t\geq 0)$ are bounded measurable 
functions.
\item[(\bf A2)] Recovery times   $\{H_y,\, y\in \Tau\}$  are given  by  independent identically distributed 
random variables. The  common distribution of recovery times is either 
absolutely continuous, or discrete, or a 
mixture of these two types of distributions.
A special case is when the recovery time is given by a deterministic constant
$H$ (including the limit case  $H=\infty$).
\end{enumerate}

\begin{remark}
{\rm 
If $\eps_t\equiv const$,  the recovery time is exponentially distributed,
the underlying graph is complete (i.e. any two vertices are neighbors)
and $\lambda_{t}\equiv 0$, 
then the corresponding discrete SIR model is a discrete version   of the 
classic Kermack-McKendrick  model. 
}
\end{remark}

\begin{remark}
{\rm 
Note that the case when   the recovery time  is given by a deterministic constant $H$
can be  modeled by assuming  that 
 $\eps_t=0$ for $t\geq H$.
 Setting  formally  $H=\infty$ gives the model, in which 
 an infected individual never recovers and stays infectious forever. This is not entirely realistic.
However, by considering a function $\eps$, which decays to zero sufficiently fast, 
 one can model a situation, when  contagiousness of  a chronically  infected individual practically vanishes
in a finite  time.
}
\end{remark}

\subsection{Distribution of the time to infection}
\label{disc-main}

In this section we state and prove the main result (Theorem~\ref{T1} below)
for the discrete SIR model on a homogeneous tree.

Let $\varphi_t$ be the probability that an infected at time $0$ vertex 
does not emit a germ towards a given neighbor until time 
 $t$, 
 and let $f_t$ be the probability that an initially susceptible vertex 
is not self-infected  until  time $t$.  Then
\begin{equation}
\label{phi-ran-f}
\varphi_t=
\begin{cases}
1,& t<0,\\
\E\left(e^{-\int_0^{t\wedge H}\eps_udu}\right), & t\geq 0,
\end{cases}
\quad\text{and}\quad
f_t=
\begin{cases}
1,&  t<0,\\
e^{-\int _0^t\lambda_u du},& t\geq 0,
\end{cases}
\end{equation}
where, for technical convenience, we define both probabilities by unity for $t<0$, and 
 $H$ is a random variable that has the same distribution as the recovery times.

\begin{theorem}
\label{T1}
Let $\Tau$  be a homogeneous tree with the vertex  degree $n+1$, where $n\geq 1$.
Assume that at time $t=0$  a  vertex  $x\in \Tau$ is infected   with probability $p$
and is susceptible with probability $1-p$ independently of other vertices.
Let $\tau$ be the time it takes for a susceptible vertex to get infected. 
Then  
\begin{equation}
\label{eqt1}
\P\left(\tau>t\right)=(1-p)f_{t}[s_{t}]^{n+1}\quad\text{for}\quad t\geq 0, 
\end{equation}
where the function $s_{t}$ 
satisfies the integral equation
\begin{equation}
\label{int_eq}
s_{t}=\varphi_{t}-(1-p)\int _0^{t} f_us^n_u\varphi'_{t-u}du, 
\end{equation}
where  functions $\varphi$ and $f$ are defined in~\eqref{phi-ran-f}
and $\varphi'$ is  the derivative of  $\varphi$. 
\end{theorem}

\begin{proof}[Proof of Theorem \ref{T1}]

Consider a {\it susceptible} vertex $x\in\Tau$ and define the probability
\begin{equation}
\label{s-def0}
s_{t}=\P(x\text{ not infected by a given neighbor before time } t), 
\end{equation}
which does not depend on a neighbor due to homogeneity of both the tree 
and the initial condition.
Recall that infectious neighbors infect  the vertex $x$ independently of each other, 
and there is an independent  chance of self-infection.
Therefore,  we have that
\begin{equation}
\label{eqt0}
\begin{split}
&\P(\text{susceptible } x\text{ not infected neither by its neighbors, nor by itself before time } t)\\
&=
f_{t}s_{t}^{n+1}\quad\text{for}\quad t\geq 0,
\end{split}
\end{equation}
and, hence, 
\begin{equation}
\label{eqt11}
\P(\tau>t)=(1-p)f_{t}s_{t}^{n+1}\quad\text{for}\quad t\geq 0,
\end{equation}
where the factor $f_{t}$ is defined in~\eqref{phi-ran-f}.

Further,  recall that  a rooted homogeneous tree with the vertex degree $N\geq 2$
as  a tree, where one of the vertices, called the root,
has $N-1$ neighbors, while any other vertex has $N$  neighbors.
Then,  observe that 
removing the vertex $x$ and all edges connecting $x$ to its neighbors generates 
$n+1$ subgraphs given by rooted trees  with  roots $y_1,...,y_{n+1}$,
that are neighbors of $x$.
Further, given  a neighbor    $y\sim x$ consider an auxiliary SIR model 
on the rooted tree $\Tau_y$  with the root $y$. 
Assume that   the auxiliary  SIR model is specified by the same parameters as the original SIR model on 
the tree $\Tau$.  In particular, 
we assume  that at time $t=0$  any vertex in the auxiliary model is either infected with probability $p$ or susceptible with probability $1-p$,   independently of other vertices. 
Let $\widetilde\tau$ be  the time to infection of the  root vertex  $y$  in the auxiliary SIR model
on the graph $\Tau_y$. Similarly to equations~\eqref{eqt0}-\eqref{eqt11} in the original model, we have
that
\begin{equation}
\label{eqt111}
\P(\widetilde\tau>t)=(1-p)f_{t}s^n_t\quad\text{for}\quad t\geq 0.
\end{equation}
Further, observe that if the root $y$ is not infected, then, due to the  similarity of the rooted trees, 
 the time before infection of any of its susceptible  neighbor has the {\it same} distribution
as the infection time $\widetilde\tau$.
Combining this fact with  the law of  total probability with~\eqref{eqt111} gives 
the following
equation for the probability $s_t$
 \begin{equation}
\label{s-def}
\begin{split}
s_{t}&=p\varphi_{t}-(1-p)\int _0^{t}\varphi_{t-u}(f_us_u^n)'_udu+(1-p)f_ts_t^n
\quad\text{for}\quad t\geq 0,
\end{split}
\end{equation}
where the function $\varphi$ is defined in~\eqref{phi-ran-f}.
Integrating by parts gives  the  equation
\begin{equation}
\label{e13}
\begin{split}
s_{t}&
=\varphi_{t}-(1-p)
\int _0^{t}f_us^n_u
\varphi'_{t-u}du,
\end{split}
\end{equation}
which is  the integral equation~\eqref{int_eq}, as claimed.
\end{proof}

\begin{remark}
\label{Rem1}
{\rm
Consider the SIR model on a homogeneous tree, in which 
 recovery times are given  by i.i.d. random variables
 (including the degenerate case of deterministic recovery time).
Recall the function  $(\varphi_{t},\, t\geq 0)$ defined in~\eqref{phi-ran-f} and 
let
\begin{equation}
\label{tilde-eps}
\tilde{\eps}_t=-\frac{d}{dt}\left(\log(\varphi_{t})\right)\quad\text{for}\quad t\geq 0.
\end{equation}
It is easy to see that, as long as one is interested in the distribution of the time to infection, 
they can consider an equivalent model  with just  susceptible and infected  compartments,
in which  the recovery mechanism is somehow  embedded into the new infection rate 
given by  the  function $\tilde{\eps}_t=(\tilde{\eps}_t,\, t\geq 0)$.
For example, consider a SIR model with the infection rate $\eps_t=(\eps_t,\, t\geq 0)$ such that 
$\eps_t>0$ for all $t\geq 0$, and  the deterministic recovery time 
given by a constant $H>0$, then  
\begin{equation}
\label{tilde-eps-1}
\tilde{\eps}_t=\begin{cases}
\eps_t,&\text{ for } t\leq H,\\
0,&\text{ for } t>H.
\end{cases}
\end{equation}
Trivially,  if $H=\infty$, then $\tilde{\eps}_t=\eps_t$ for all $t\geq 0$. 
However, the function $\tilde{\eps}_t$ can differ significantly from the 
original function $\eps_t$ in the case of the random recovery time 
(e.g., see Corollary~\ref{C3} in Section~\ref{examples-SIR}).
}
\end{remark}

\begin{remark}
{\rm 
In some  special cases, the obtained  integral equation 
  is equivalent to the Bernoulli type differential equation, which can be solved analytically
(see Section~\ref{examples-SIR} for examples). 
}
\end{remark}

\section{Continuous limit of the discrete model}
\label{cont}

\subsection{The master  equation}
\label{master-sec}

In this section we analyze  integral equation~\eqref{int_eq} in the limit, as 
the tree vertex degree goes to infinity. Specifically, we show that  in this limit 
the integral equation implies an equation for the susceptible population proportion.

\begin{theorem}
\label{T2}
Consider the discrete SIR model on the homogeneous tree $\Tau$ 
with  the vertex degree $n+1$.  Suppose that  an infected vertex infects 
a susceptible neighbor with the rate $\frac{1}{n+1}\eps_t$ after time $t$ of being infected, where 
$(\eps_t,\, t\geq 0)$ is a non-negative function. In addition, suppose that the 
other model parameters (i.e. the recovery times and the rate of self-infection) do not depend on $n$.
Let $S_{t,n}$ be the expected susceptible population in this SIR model. 
Let  $(S_{t},\, t\geq 0)$ be a limit point 
 of the sequence of functions $(S_{t,n},\, t\geq 0)$,\, $n\geq 1$, in the sense of the pointwise 
convergence. Then the function $(S_{t},\, t\geq 0)$ must satisfy  the following equation
\begin{equation}
\label{master0}
\log\left(\frac{S_{t}}{S_0}\right)=-\int_0^t\lambda_udu-\int_0^t(1-S_u)\gamma_{t-u}du,
\end{equation}
where
\begin{equation}
\label{gamma}
\gamma_t:=\eps_t\P(H>t) \quad\text{for}\quad t\geq 0
\end{equation}
and $S_0=1-p$ (i.e. it is the probability for a vertex to be susceptible at time $t=0$ in the discrete model).
\end{theorem}

\begin{proof}
Note first, that  the corresponding $\varphi$-function (see equation~\eqref{phi-ran-f}) is given by 
$$ \varphi_{t,n}=
\begin{cases}
1,& t<0,\\
\E\left(e^{-\frac{1}{n+1}\int_0^{t\wedge H}\eps_udu}\right), & t\geq 0.
\end{cases}
$$
By Theorem~\ref{T1},
\begin{equation}
\label{S}
S_{t,n}=S_{0,n}f_{t}[s_{t,n}]^{n+1}\quad\text{for}\quad t\geq 0,
\end{equation}
where  the function  
 $s_{t,n}$ satisfies the   equation
\begin{equation}
\label{int_eq2}
s_{t,n}=\varphi_{t,n}-S_{0,n}\int _0^{t}f_u s_{u,n}^n\varphi'_{t-u,n}du
\quad\text{with}\quad S_{0,n}=1-p=S_0.
\end{equation}
By~\eqref{S}-\eqref{int_eq2},
$$s_{t,n}=\left(\frac{S_{t,n}}{S_{0}f_{t}}\right)^{\frac{1}{n+1}}=1+
\frac{1}{n+1}\left(\log\left(\frac{S_{t,n}}{S_{0}}\right)-\log(f_{t})\right)+o\left(\frac{1}{n}\right),$$
so that 
\begin{equation}
\label{int_eq3}
\log\left(\frac{S_{t,n}}{S_{0}}\right)-\log(f_{t})
=(n+1)(\varphi_{t,n}-1)-
\int\limits_{0}^t
\left(S_{u,n}\right)^{\frac{n}{n+1}} (S_{0}f_u)^{\frac{1}{n+1}}[(n+1)\varphi'_{t-u,n}]du.
\end{equation}
A direct computation (we skip detailes) gives that 
\begin{align}
\label{gamma1}
(n+1)(\varphi_{t,n}-1)
&=-\E\left(\int\limits_0^{t\wedge H}\eps_udu\right)+o(1)=\int_0^t\gamma_udu+
o\left(1\right)\\
(n+1)\varphi'_{t,n}&=-\gamma_t e^{-\frac{1}{n+1}\int_0^t\eps_udu}=
-\gamma_t(1+o\left(1\right))
\label{gamma2}
\end{align}
for any fixed $t\geq 0$, 
where the function $\gamma_t$ is defined in~\eqref{gamma}.
In addition, note  that $\left(S_{0}f_u\right)^{\frac{1}{n+1}}\to 1$
and $\left(S_{u,n}\right)^{\frac{n}{n+1}}\to S_{u,n}$, as $n\to \infty$.
Combining this with~\eqref{gamma1}-\eqref{gamma2}
allows to  rewrite~\eqref{int_eq3} as follows
\begin{equation}
\label{int_eq31}
\begin{split}
\log\left(\frac{S_{t,n}}{S_{0}}\right)-\log(f_{t})&= -\int_0^t\gamma_udu+
\int\limits_{0}^tS_{u,n}\gamma_{t-u}du+o\left(1\right)\\
&=
-\int\limits_{0}^t(1-S_{u,n})\gamma_{t-u}du +o\left(1\right).
\end{split}
\end{equation}
Since
$\log(f_t)=-\int_0^t\lambda_udu$, 
we obtain that 
\begin{equation}
\label{int_eq32}
\log\left(\frac{S_{t,n}}{S_{0}}\right)=-\int_0^t\lambda_udu
-\int\limits_{0}^t (1-S_{u,n})\gamma_{t-u}du+o\left(1\right),
\end{equation}
which  implies, by the dominated convergence theorem,
 equation~\eqref{master0} for any pointwise 
limit point $(S_t,\, t\geq 0)$ for the sequence of functions $(S_{t,n},\, t\geq 0)$,\, $n\geq 1$, as claimed.
\end{proof}
Differentiating~\eqref{master0}  gives  the master equation in the differential form 
\begin{equation}
\label{master-initial}
\frac{S_t'}{S_t}=-\lambda_{t}-(1-S_t)\gamma_{0}-\int_0^t(1-S_u)\gamma'_{t-u}du,
\end{equation}
or, equivalently, 
\begin{equation}
\label{master}
\frac{S_t'}{S_t}=-\lambda_{t}-(1-S_0)\gamma_{t}+\int_0^tS'_u\gamma_{t-u}du.
\end{equation}

\begin{remark}
{\rm 
The existence and the uniqueness of solution of equation~\eqref{master0} follows from general results 
for  integral equations with delay (\cite{Burton}).
It can be shown that
the sequence of functions $(S_{t,n},\, t\geq 0)$,\, $n\geq 1$, is equicontinuous. Therefore, there exists 
a subsequence that uniformly  converges  to the solution of~\eqref{master0}. We skip the
 technical details.
}
\end{remark}

\begin{remark}
\label{RemStability}
{\rm 
It follows from equation~\ref{master0} that stationary 
 value $S_\infty$ satisfies the following equation
\begin{equation}
\label{S_inf}
\log\left(\frac{S_{\infty}}{S_0}\right)= -\int_0^\infty\lambda_u du- (1-S_\infty) \int_0^\infty\gamma_{u}du.
\end{equation}
}
\end{remark}

\begin{remark}
{\rm 
Note that equation~\eqref{master0} (or its differential equivalent~\eqref{master}) 
 is a standalone equation for 
the  susceptible population  $S_t$, namely 
that this equation  does not involve  neither the infected, nor the recovered 
populations. 
}
\end{remark}

\begin{remark}
\label{rem-on-gamma(t)}
{\rm 
It should be noted that   all the  information concerning the infection rates and recovery times 
of the original discrete SIR model is included in~\eqref{master} via  the   function
$\gamma$.
For example, if the recovery time in the discrete SIR model  is given by a deterministic constant $H$, then 
$$
\gamma_{t}=\begin{cases}
\eps_t,& \text{for } t<H,\\
0, &\text{for } t\geq H,
\end{cases}
$$
where $\eps_t$ is the rate of infection in the discrete model. 
In particular, if $H=\infty$, then $\gamma_{t}=\eps_t$.
In general, these two  functions  are different
(see Example~\ref{example-SIR} below).
}
\end{remark}

\subsection{Continuous SIR models implied by the master equation}
\label{special-continuous}

In this section we show that  the master equation~\eqref{master0}
for  the susceptible population implies equations for  
other population compartments (which explains 
the term master equation). 
A SIR model implied by the master equation~\eqref{master0}
depends on the structure and its interpretation  
of the function $\gamma$. To clarify what is meant by "interpretation"
consider the case when $\gamma_t=0$ for all $t>H$ for some $H>0$.
This can be interpreted as an infected individual recovering after time $H$ since the moment 
of being infected (as in  Remark~\ref{rem-on-gamma(t)}). 
On the other hand,  this can be  interpreted,  as if an infected individual 
never recovers, but becomes not contagious to others after time 
$H$ since the moment of being infected. 
In this case one can operate with just two compartments, namely susceptible and infected ones.
Below we consider examples, where this argument is reinforced.
Note that for simplicity of exposition  and without loss of generality 
we assume throughout this section  that
\begin{equation}
\label{lambda=0}
\lambda_{t}\equiv 0,
\end{equation}
 i.e. there is no self-infection.

\subsubsection{The model with no recovery}
\label{no-recovery}

The basic continuous SIR model implied by the master 
equation is a two-compartmental model, in which 
the population is divided into two compartments,  namely,
the compartment of susceptible individuals, described 
by the variable $S_t$, and the compartment of infected ones, described by
the variable $I_t$, so that
\begin{equation}
\label{balance}
1=S_t+I_t\quad\text{and}\quad\text{for}\quad t\geq 0.
\end{equation}
Then  $S'_t=-I'_t$, which allows to rewrite 
equation~\eqref{master} as follows 
\begin{equation}
\label{equations-no-recovery}
S'_t=-S_t\left(I_0\gamma_{t}
+\int_0^tI'_u\gamma_{t-u}du\right).
\end{equation}
Equation~\eqref{equations-no-recovery} describes the model, in which 
 an individual  infected at time $u\geq 0$ infects 
any susceptible individual with the rate $\gamma_{t-u}$ at time $t>u$.
This model can be interpreted as the model without recovery.

\begin{example}[The model with latent period]
\label{latent}
{\rm 
Suppose that an infected individual is  latent 
for a non-random period of time of length $L>0$.
In addition, suppose  that the rate of infection is constant.  
Then
 $\gamma_t=\eps{\bf 1}_{\{t\geq L\}}$, where  $\eps$ is the rate of infection,
and equation~\eqref{equations-no-recovery} becomes 
as follows
$$
S_t'
=-\eps S_tI_{t-L}\quad\text{and}\quad I_t'=-S'_{t}.
$$
}
\end{example}

\subsubsection{Model with a constant rate of infection  and random recovery}
\label{constant-eps}
Suppose that the function  $\gamma_t$ is 
of  the following form
\begin{equation}
\label{gamma1}
\gamma_{t}=\eps\beta_{t}\quad\text{for}\quad t\geq 0,
\end{equation}
where  $\eps>0$ is a given constant and 
$\beta_{t}$ is a non-increasing positive function, such that $\beta_{0}=1$ and 
 $\beta_t\to 0$, as $t\to \infty$.
Then, the master equation implies the continuous SIR  model with the three standard
compartments, in which 
an infected  individual
 recovers  in a  time given by  random variable $\xi$ with the tail distribution
$\P(\xi>t)=\beta_{t}$,   and during  its infectious 
period it  infects any susceptible one with the constant rate $\eps$. 
Indeed, under these assumptions, equation~\eqref{master}
is as follows (recall that~\eqref{lambda=0})
\begin{equation}
\label{master-beta}
S_t'=-\eps S_t\left(I_0\beta_{t}-\int\limits_0^tS'_u\beta_{t-u}du\right).
\end{equation}
Define 
\begin{align}
\label{I_t}
I_t&=I_0\beta_{t}-\int\limits_0^tS'_u\beta_{t-u}du\quad\text{for}\quad t>0\quad\text{and}\quad
 I_0=1-S_0
\end{align}
and
\begin{align}
\label{R}
R_t&=I_0(1-\beta_{t})-\int\limits_0^tS'_u(1-\beta_{t-u})du
\quad\text{for}\quad t>0\quad\text{and}\quad
 R_0=0.
\end{align}
It is easy to see that 
\begin{equation}
\label{balance2}
1=S_t+I_t+R_t\quad\text{for}\quad t\geq 0.
\end{equation}
Moreover, one can show  that both $I_t\geq 0$ and $R_t\geq 0$ (we skip the details).
Therefore, variables $I_t$ and $R_t$ can be interpreted 
as the  population proportions of infected and recovered 
individuals respectively in the continuous SIR model with the constant rate of infection $\eps$
and the random recovery time with the tail distribution given by the function $\beta$.
Indeed, in this model the infected compartment at time $t$  consists of 
1) those who were  infected at time $0$ and did  not recover before time $t$
(which gives the first term 
$I_0\beta_{t}$ in~\eqref{I_t}), and 
2) those, who were  infected at time $u\in (0, t]$ and did not recover before
 time $t$ (integrating over time gives  the integral term in~\eqref{I_t}).
A similar argument gives~\eqref{R1}, which also follows from~\eqref{I_t} and~\eqref{balance2}
combined with the initial condition $S_0+I_0=1$.
Differentiating both~\eqref{I_t} and~\eqref{R}, and combining them with~\eqref{master-beta},
 we get   the following  system of equations 
\begin{align}
\label{S-eq}
S'_t&=-\eps S_t I_t,\\
\label{I-eq}
I'_t&=\eps S_tI_t+I_0\beta'_{t}+\eps\int\limits_0^tS_uI_u\beta'_{t-u}du,\\
\label{R-eq}
R'_t&=-I_0\beta'_{t}-\eps\int\limits_0^tS_u I_u\beta'_{t-u}du.
\end{align}

\begin{remark}
\label{Anna}
{\rm 
The system of equations~\eqref{S-eq}-\eqref{R-eq} is similar to the   
system of equations  of the delay model proposed 
 in~\cite{DellAnna} for modeling the spread of Covid-19 in Italy.
}
\end{remark}

\begin{example}
[The classic SIR model]
\label{example-SIR}
{\rm 
Consider the discrete SIR model on a homogeneous tree with the vertex degree $n+1$. 
Assume that the infection rate is constant, i.e.  $\eps_t\equiv \eps$ for some constant $\eps>0$, and that 
the recovery time $H$ is exponentially distributed with parameter $\mu$, i.e. $\P(H>t)=e^{-\mu t}$ for $t\geq 0$.
In addition, assume that there is no self-infection, i.e. 
 $\lambda_{t}=0$ for $t\geq 0$.
Then 
$$\E\left(\int_0^{t\wedge H}\eps_udu\right)=\eps\E(\min(t, H))=\frac{\eps}{\mu}\left(1-e^{-\mu t}\right),$$
so that 
$\gamma_{t}=\eps e^{-\mu t}$ for $t\geq 0$. 
Setting  
$\beta_{t}=\frac{\gamma_{t}}{\eps}=e^{-\mu t}$ gives a special case of~\eqref{gamma1}.
The system of equations~\eqref{S-eq}-\eqref{R-eq} becomes 
as follows
\begin{align}
\label{S-SIR}
S_t'&=-\eps S_t I_t,\\
\label{I-SIR}
I_t'&=\eps S_t I_t-\mu  I_t,\\
\label{R-SIR}
R_t'&=\mu  I_t,
\end{align}
which is the system of equations of the  classic Kermack-McKendrick  model (with the infection rate $\eps$ and the recovery rate $\mu$).
}
\end{example}

\begin{remark}
\label{remark-SIR}
{\rm 
Note that  in   Example~\ref{example-SIR}
it  is probably more convenient  to start  with 
equation~\eqref{master0}, which in this case is   as follows
\begin{equation}
\label{int_exp1}
\log\left(\frac{S_t}{S_0}\right)=-\eps\int_0^t(1-S_u)e^{-\mu(t-u)}du.
\end{equation}
Then, differentiating~\eqref{int_exp1} gives  that
\begin{equation}
\label{basic1}
S_t'=-\eps S_t\left(1-S_t-\mu\int_0^t(1-S_u)e^{-\mu(t-u)}du\right).
\end{equation}
Combining~\eqref{basic1} with~\eqref{int_exp1} we obtain 
the following  equation 
\begin{equation}
\label{basic2}
S_t'=-\eps S_t\left(1-S_t+\frac{\mu}{\eps}\log\left(\frac{S_t}{S_0}\right)\right),
\end{equation}
which is the master equation (in the differential form) 
 corresponding to the Kermack-McKendrick model.
Setting $I_t=1-S_t+\frac{\mu }{\eps}\log\left(\frac{S_t}{S_0}\right)$, one can proceed as 
in Example~\ref{example-SIR} to get the equations~\eqref{S-SIR}-\eqref{R-SIR}.
}
\end{remark}

\begin{remark}
\label{remark-Harko}
{\rm 
It should be noted that the master  equation~\eqref{basic2} is well-known. For example, it is the same as 
 equation (22) in \cite{Kroger} and  is also equivalent to equation (26) in~\cite{Harko}.
}
\end{remark}

\begin{remark}
{\rm 
Note that  equating the time derivative to zero in equation~\eqref{basic2}, i.e. 
$S'_t=0$, gives the known equation  
\begin{equation}
\label{stat}
1-S+\frac{\mu }{\eps} \log\left(\frac{S}{S_0}\right)=0
\end{equation}
for  the stationary  population proportion of susceptible individuals $S$ in the SIR model
(e.g. see equation (7) in~\cite{Barlow} and references therein).
In particular, this  equation  shows that  the stationary value $S$ depends only on the ratio 
 $\displaystyle{\mu/\eps=1/R_0}$, where $R_0$ is the basic reproduction number in the classic SIR model.
Note also that equation~\eqref{stat} is just a special case of more general equation~\eqref{S_inf} 
for the stationary susceptible state.
}
\end{remark}

\begin{example}[Constant rate of infection and deterministic recovery]
{\rm 
Consider a model, in which 
the infection rate is given by a constant $\eps>0$, and the recovery time is given 
by a deterministic constant $H>0$.
This model can be obtained by setting 
$\gamma_t=\eps$ for $t\in [0,H]$ and $\gamma_t=0$ for $t>H$.
This gives the  following model 
equations
\begin{align*}
S_t'&=-\eps S_tI_t,\\
I_t'&=\eps S_tI_t-\eps S_{t-H}I_{t-H},\\
R_t'&=\eps S_{t-H}I_{t-H},
\end{align*}
where $S_t=I_t=0$ for $t<0$. 
}
\end{example}

\subsubsection{The general case: time-varying infection rate and random recovery}
\label{general-SIR}

In this section we generalize SIR 
models considered in Sections~\ref{no-recovery} and~\ref{constant-eps}.

Suppose that the function $\gamma$ is of the following form 
\begin{equation}
\label{gamma-general}
\gamma_{t}=w_t\beta_{t}\quad\text{for}\quad t\geq 0,
\end{equation}
where $(w_t,\, t\geq 0)$ is a non-negative function 
and the function $(\beta_t,\, t\geq 0)$ is the tail distribution 
of some positive random variable (i.e. similarly to what  we assumed 
 in Section~\ref{constant-eps}).
Arguing as in Section~\ref{constant-eps}, we obtain 
the  continuous SIR model
described by the following equations
\begin{align}
\label{master03}
\log\left(\frac{S_{t}}{S_0}\right)&=-\int_0^t(1-S_u)\gamma_{t-u}du,\\
\label{I_t-30}
I_t&=I_0\beta_{t}-\int\limits_0^tS'_u\beta_{t-u}du,\\
\label{R1}
R_t&=I_0(1-\beta_{t})-\int\limits_0^tS'_u(1-\beta_{t-u})du,
\end{align}
where,  $S_t$, $I_t$ and $R_t$ are population proportions 
of susceptible, infected and recovered individuals respectively, so that 
$1=S_t+I_t+R_t$ for $t\geq 0$.  As before,  we assumed that $R_0=0$.
In this model  an infected individual  recovers 
in a random time given by a random variable $\xi$ with the tail distribution 
$\P(\xi>t)=\beta_{t}$  and, if it is infected at time $u$, then it  
 infects any susceptible one with the rate $w_t$ at the time $u+t$.
Recall, that we also assume~\eqref{lambda=0}.

In the differential form the model equations are as follows
\begin{align}
\label{master3}
S_t'&=S_t\left(-I_0\gamma_{t}+\int_0^tS'_u\gamma_{t-u}du\right),\\
\label{I3}
I_t'&=-S_t'+I_0\beta'_{t}-\int_0^tS'_{u}\beta'_{t-u}du,\\
\label{R3}
R_t'&=-I_0\beta'_{t}+\int_0^tS'_{u}\beta'_{t-u}du.
\end{align}

\begin{remark}
{\rm 
By choosing appropriate functions  $(w_t,\, t\geq 0)$ and $(\beta_t,\, t\geq 0)$ one can model various 
infection rates and recovery distributions.
For example, using the power law functions allows 
to  model memory  effects observed in real data (e.g. see~\cite{Angstmann}
and references therein). 
}
\end{remark}

In the rest of this section we use the idea from~\cite{Angstmann} in order 
to rewrite equations~\eqref{master3}-\eqref{R3}
in terms of a certain kernel.  The idea is based on the fact  
that these equations contain convolutions, which makes  it possible to apply the Laplace transform.

Let  $(\L\{g\}_t,\, t\in \R_{+})$ be the Laplace transform of a function $(g_t,\, t\in \R_{+})$.
It follows from~\eqref{I3} that
$$\L\{I\}_t=I_0\L\{\beta\}_t-\L\{S'\}_t\L\{\beta\}_t,$$
and, hence, 
$$\L\{S'\}_t=I_0-\frac{\L\{I\}_t}{\L\{\beta\}_t}.$$
Thus, for any  appropriate function $(g_t,\,t\in\R_{+})$ we have that 
\begin{equation}
\label{fact0}
\begin{split}
\int_0^tS'_{u}g_{t-u}du&=\L^{-1}\left[\L\{S'\}_t\L\{g\}_t\right]\\
&=
\L^{-1}\left[\left(I_0-\frac{\L\{I\}_t}{\L\{\beta\}_t}\right)\L\{g\}_t\right]
=I_0g_t -\L^{-1}\left(\L\{I\}_t\frac{\L\{g\}_t}{\L\{\beta\}_t}\right).
\end{split}
\end{equation}
Since $\L^{-1}(\L\{a\}\L\{b\})$ is equal to the convolution $a\ast b$, we can rewrite~\eqref{fact0} as follows
\begin{equation}
\label{fact}
\begin{split}
\int_0^tS'_{u}g_{t-u}du&=I_0g_t-\int_0^tI_{u}\K(g)_{t-u}du,
\end{split}
\end{equation}
where $\K$ is a  kernel  defined by
\begin{equation}
\label{K}
\mathcal{K}(g)_t:=\L^{-1}\left(\frac{\L\{g\}_t}{\L\{\beta\}_t}\right).
\end{equation}
Finally, using~\eqref{fact} with $g_t=-\gamma_t$ in~\eqref{master3},  
and with $g_t=-\beta'_t$ in~\eqref{I3} and~\eqref{R3} 
gives the system of the model equations in the kernel form
\begin{align}
\label{master31}
S_t'&=-S_t\int_0^tI_{u}\K(\gamma)_{t-u}du,\\
\label{I31}
I_t'&=S_t\int_0^tI_{u}\K(\gamma)_{t-u}du-\int_0^tI_{u}\K(\beta')_{t-u}du,\\
\label{R31}
R_t'&=\int_0^tI_{u}\K(\beta')_{t-u}du.
\end{align}

\begin{remark}
{\rm 
It should be noted  that  equation~\eqref{K} 
is an analogue of  equation (16) in~\cite{Angstmann}. 
}
\end{remark}

\section{Remark on fractional SIR models}
\label{fractional}

One of the recognized drawbacks  of the classic SIR model is that 
  both the infection rate and the recovery rate do  not depend on the state of the system, i.e. 
the model is memoryless.
A popular approach to modeling memory effects consists in using fractional SIR models
(e.g., see ~\cite{Chen} and references therein). 
Some of these  models are obtained 
by formal replacement of ordinary derivatives by fractional derivatives of a certain type.
This  gives a system of fractional differential equations
that is equivalent to  a system of  integro-differential equations 
with a power-law kernel.
For example, replacing ordinary derivatives in the classic SIR model by Caputo  fractional derivatives
gives a system of  fractional differential equations, which are 
 equivalent to the following system of integro-differential equations 
\begin{align}
\label{formal}
S'_t&=-\eps \int _0^tI_{u}S_{u}K_{t-u}du,\\
\label{I-frac}
I'_t&=\int _0^t\left(\eps  I_{u}S_{u}-\mu I_{u}\right)K_{t-u}du,\\
\label{R-frac}
R'_t&=\mu  \int _0^tI_{u}K_{t-u}du,
\end{align}
with the kernel $K_y=\frac{y^{\alpha -2}}{\Gamma (\alpha -1)}$, where 
$\alpha\in (0,1]$ and  
 $\Gamma$ is the Gamma-function. 
However,  it is not quite clear  what physical/biological process 
is described by equations~\eqref{formal}-\eqref{R-frac}.
In contrast, equations~\eqref{master3}-\eqref{R3} and their equivalents
in the kernel form, i.e. equations~\eqref{master31}-\eqref{R31}, 
can be naturally interpreted in terms of the interaction between compartments.
Indeed, using the terminology of~\cite{Angstmann}, one can say that, 
for example, 
the flux  into the infected  compartment is equal to the  flux out of the 
susceptible compartment (in the absence of any external factors and self-infection). 
The susceptible compartment decreases at the rate 
proportional to its current value $S_t$. 
The  value $\K(\gamma)_{t-u}$ (in~\eqref{master31}-\eqref{R31})
  describes the impact made on the  susceptible compartment at time $t$
by those individuals, who were infected earlier and is still infectious.
The coefficient of proportionality, i.e.  the  integral term 
$\int_0^tI_{u}\K(\gamma)_{t-u}du$,  measures
the total impact of the infected compartment on the susceptible one over the time period 
$[0,t]$. This generalizes the interaction between  susceptible and  infected 
compartments in the classic SIR model, where only the current value $I_t$ is 
taken into account.

It should be also noted that the continuous SIR model in the present paper 
is obtained by passing to the limit in the discrete stochastic SIR model. 
This is in line with SIR models in the kernel form that are derived from stochastic processes based on 
natural biological assumptions  (e.g., see~\cite{Angstmann},~\cite{DellAnna} and references therein).

\section{Appendix. Special  cases of the discrete SIR model}
\label{examples-SIR}

In this  section, we consider some special cases of the discrete SIR model
on the homogeneous  tree with the vertex degree $n+1$. 
In these cases the integral equation~\eqref{int_eq}
 can be rewritten in an equivalent differential form, which is of interest on its own right.

For simplicity of notations we assume that all vertices are  initially  susceptible
(i.e. $p=0$ in Theorem~\ref{T1}).
\begin{corollary}
\label{C1}
Suppose that there is no recovery, i.e. $H=\infty$,  
 $\eps_t=\eps\bi_{\{t\geq 0\}}$ and $\lambda_{t}=\lambda\bi_{\{t\geq 0\}}$, where $\eps>0$ and 
$\lambda>0$ are given constants.
Then 
\begin{equation}
\begin{split}
\label{exact}
s_{t}&=e^{-\frac{2\eps}{\lambda}\left(e^{-\lambda t}-1+\lambda t\right)},
\quad\text{if}\quad n=1;\\
s_{t}&=\left(\frac{\eps(n-1)+\lambda}{\eps(n-1)
e^{-\lambda t}+\lambda e^{\eps(n-1)t}}\right)^{\frac{1}{n-1}},\quad\text{if}\quad n\geq 2,
\end{split}
\end{equation}
so that 
\begin{align}
\label{n1}
\P(\tau>t)&=e^{-\lambda t}e^{-\frac{2\eps}{\lambda}\left(e^{-\lambda t}-1+\lambda t\right)},
\quad\text{if}\quad n=1;\\
\P(\tau>t)&
=e^{-\lambda t}\left(\frac{\eps(n-1)+\lambda}{\eps(n-1)e^{-\lambda t}+\lambda e^{\eps(n-1)t}}\right)^{\frac{n+1}{n-1}},\quad\text{if}\quad n\geq 2.\label{nn}
\end{align}
\end{corollary}

\begin{proof}[Proof of Corollary~\ref{C1}]

Since $\eps_t=\eps\bi_{\{t\geq 0\}}$ and 
$\lambda_{t}=\lambda\bi_{\{t\geq 0\}}$, we have  that 
\begin{equation}
\label{C1P1}
\varphi_{t}=
\begin{cases}
1, & t<0,\\
e^{-\eps t},& t\geq 0,
\end{cases}
\quad\text{and}\quad
f_{t}=\begin{cases}
1, & t<0,\\
e^{-\lambda t},& t\geq 0,
\end{cases}
\end{equation}
and  
 $\varphi'_t=-\eps e^{-\eps t}$ for $t\geq 0$. The 
 integral equation~\eqref{int_eq} in this case is as follows
\begin{equation}
\label{int_eq0}
s_{t}=e^{-\eps t}+\eps \int _0^te^{-\lambda u}s^n_ue^{-\eps(t-u)}du.
\end{equation}
Differentiating~\eqref{int_eq0} and simplifying gives the following differential equation
\begin{equation}
\label{ber1}
\begin{split}
s'_t
=-\eps s_{t}+\eps e^{-\lambda t}s^n_t.
\end{split}
\end{equation}
It is easy to verify  that the function defined in~\eqref{exact} is a solution of  equation~\ref{ber1}.
\end{proof}

\begin{remark}
{\rm 
The equation~\eqref{ber1} is  the well known   Bernoulli equation (e.g., see~\cite{Parker}).
Under assumptions of  Corollary~\ref{C1} the model was originally considered in~\cite{Gairat}.
}
\end{remark}

\begin{remark}
\label{R-logistic}
{\rm 
It follows from equation~\eqref{nn}
that \begin{align*}
\P(\tau>t)&=e^{-\lambda t}\left(\frac{\eps(n-1)+\lambda}{\eps(n-1)e^{-\lambda t}+\lambda e^{\eps(n-1)t}}\right)^{\frac{n+1}{n-1}}\\
&\sim e^{-\lambda t}\left(\frac{\eps(n-1)+\lambda}{\eps(n-1)e^{-\lambda t}+\lambda e^{\eps(n-1)t}}\right)
=
\frac{1+\frac{\lambda}{\eps(n-1)}}{1+\frac{\lambda}{\eps(n-1)}e^{(\eps(n-1)+\lambda)t}}
\end{align*}
for sufficiently large $n$. 
Further, if $\eps=\eps_n$, where 
$\eps_nn\to c>0$, as $n\to\infty$, 
then 
\begin{equation}
\label{Eq-logistic}
1-\P(\tau>t)=\P(\tau\leq t) \to
 \frac{\frac{\lambda}{c} e^{(c+\lambda)t}-\frac{\lambda}{c}}{1+\frac{\lambda}{c} e^{(c+\lambda)t}}
 =\left(1+\frac{\lambda}{c}\right)\frac{1}{1+\frac{c}{\lambda} e^{-(c+\lambda)t}}-\frac{\lambda}{c},
\quad\text{as}\quad n\to \infty.
\end{equation}
 In other words,  the probability $\P(\tau\leq t)$ converges
to a linear transformation of the 
logistic curve $\displaystyle{\frac{1}{1+\frac{c}{\lambda}e^{-(c+\lambda)t}}}$.
 }
\end{remark}

\begin{corollary}
\label{C2}
Suppose that the recovery time is given by a deterministic constant $H>0$,
$\lambda_{t}=\lambda{\bf 1}_{\{t\geq 0\}}$ 
 and  $\eps_{t}=\eps\bi_{\{0\leq t\leq H\}}$, 
 where $\eps>0$ and $\lambda>0$ are given constants.
Then integral equation~\eqref{int_eq} is equivalent to the following differential equation
\begin{equation}
\label{ber2}
\begin{split}
s'_t&=-\eps  s_{t}+\eps e^{-\lambda t}s^n_t\quad\text{for}\quad t\leq H,\\
s'_t&
=-\eps s_{t}+\eps e^{-\lambda t}s^n_t+\eps e^{-\eps H}-\eps e^{-\lambda(t-H)-\eps H}s^n_{t-H}
\quad\text{for}\quad t>H.
\end{split}
\end{equation}
\end{corollary}

\begin{proof}[Proof of Corollary~\ref{C2}]
In this case we have that  
$f_{t}=e^{-\lambda t}$ and   $\varphi_{t}=e^{-\eps\min(t, H)}$ for $t\geq 0$,
and equation~\eqref{int_eq} becomes as follows
\begin{equation}
\label{sH}
s_{t}=\begin{cases}
\varphi_{t}+\eps\int_0^te^{-\lambda u}s^n_t\varphi_{t-u}du&\text{for}\quad t<H,\\
\varphi_{t}+\eps \int_{t-H}^te^{-\lambda u}s^n_u\varphi_{t-u}u&\text{for}\quad t\geq H.
\end{cases}
\end{equation}
A direct computation gives that 
$$
s'_t=-\eps\varphi_t+\eps  e^{-\lambda t}s^n_t+\eps ^2\int _{0}^te^{-\lambda u}s^n_u\varphi_{t-u}du
=-\eps s_{t}+\eps e^{-\lambda t}s^n_t\quad\text{for}\quad t<H,
$$
which is the first equation in~\eqref{ber2}. 

Further, if $t>H$, then  $\varphi_{t}=\varphi_H=e^{-\eps H}$, 
and, hence, 
\begin{align*}
s'_t&=\eps e^{-\lambda t}s^n_t-\eps  e^{-\lambda(t-H)-\eps H}s^n_{t-H}
-\eps ^2\int _{t-H}^tf(u)s^n_u\varphi_{t-u}du\\
&=\eps e^{-\lambda t}s^n_t-\eps  e^{-\lambda(t-H)-\eps H}s^n_{t-H}-\eps(s_{t}-\varphi_H)\\
&=-\eps  s_{t}+\eps  e^{-\lambda t}s^n_t+\eps e^{-\eps H}-\eps  e^{-\lambda (t-H)-\eps H}s^n_{t-H},
\end{align*}
and this is  the second  equation in~\eqref{ber2}, as claimed.
\end{proof}

We are going to consider an example of the model  with a random recovery  time.

\begin{corollary}[Exponential recovery time]
\label{C3}
Suppose that the  recovery time is exponentially distributed with the  parameter $\mu$
and functions  $(\eps_{t},\, t\in\R_+)$ and
  $(\lambda_{t},\, t\in\R_+)$ 
are  as  in  Corollaries~\ref{C1} and~\ref{C2}.
Then integral equation~(\ref{int_eq}) is equivalent to the following differential equation 
\begin{equation}
\label{free-term}
s'_t=-(\mu+\eps)s_{t}+\eps e^{-\lambda t}s^n_t+\mu.
\end{equation}
\end{corollary}

\begin{proof}[Proof of Corollary~\ref{C3}]
Start with computing  the corresponding function $\varphi$   
 \begin{equation}
 \label{expon}
\begin{split}
 \varphi_{t}=\E\left(e^{-\eps \min(t, H)}\right)
  &=
\mu\int_0^te^{-\eps  u}e^{-\mu u}du+\mu e^{-\eps  t}\int_t^{\infty } e^{-\mu u}du\\
&
=\frac{\mu}{\mu+\eps}+\frac{\eps}{\mu+\eps}e^{-(\mu+\eps)t}
\quad\text{for}\quad t\geq 0.
\end{split}
\end{equation}
Recall Remark~\ref{Rem1} and define 
\begin{equation}
\label{d=ran}
\tilde{\eps}_{t}:=-\left(\log(\varphi_t)\right)'=
\frac{\eps  (\mu +\eps )}{\mu  e^{(\mu +\eps )t}+\eps }\bi_{\{t\geq 0\}}\quad\text{for}\quad t\geq 0,
\end{equation}
so that 
$\varphi'_t=-\tilde{\eps}_{t}\varphi_{t}$ for $t\geq 0$.
A direct computation gives that 
$$\tilde{\eps}'_t-\tilde{\eps}^2_t
=-(\mu +\eps )\widetilde\gamma_{t}\quad\text{for}\quad t>0.$$
 Using the equation in the  preceding  display and 
differentiating equation~\eqref{int_eq} gives that 
\begin{align*}
s'_t
&=-\tilde{\eps}_{t}\varphi (t)+\eps  f_{t}s^n_t+\int _0^{t}f(u)s^n_u\varphi_{t-u}
\left(\tilde\eps_{t-u}'-\eps^2_{t-u}\right)du
\\
&=-\tilde{\eps}_{t}\varphi_t+\eps  f_{t}s^n_t-
(\mu +\eps)\int _0^{t}f_us^n_u\varphi_{t-u}\tilde{\eps}_{t-u}du
\\
&=-\tilde{\eps}_{t}\varphi (t)+\eps  f_{t}s^n_t-(\mu +\eps)(s_{t}-\varphi_t)\\
&=-(\mu +\eps)s_{t}+\eps  f_{t}s^n_t+(\mu +\eps -\tilde{\eps}_{t})\varphi_t.
\end{align*}
Noting  that 
$$
(\mu +\eps -\tilde{\eps}_{t})\varphi_t=\left(\mu +\eps-\frac{\eps(\mu+\eps)}{
\mu e^{(\mu+\eps)t}+\eps}\right)\left(\frac{\mu}{\mu+\eps}+\frac{\eps}{\mu+\eps}e^{-(\mu+\eps)t}\right)=\mu,
$$
gives the equation $s'_t=-(\mu +\eps)s_{t}+\eps  f_{t}s^n_t+\mu$, which is 
the  Bernoulli equation with the additional term $\mu$, 
as claimed.
\end{proof}


\begin{thebibliography}{1}
\bibitem{Andersson} 
Andersson,  H., and  Britton, T. ({\bf 2000}). 
Stochastic epidemic models and their statistical analysis. 
 {\it Lecture Notes in Statistics}, {\bf 151}, Springer, New York.
\bibitem{Angstmann}  Angstmann, C.N., Henry, B. I., and McGann, A.V. ({\bf 2016}).
A Fractional Order Recovery SIR Model from a Stochastic Process.
{\it Bulletin of Mathematical  Biology}, {\bf 78}, pp.~468--499.
\bibitem{BrittonLN} Ball, F., Britton, T., Larédo, C.,  Pardoux, E.,  Sirl, D.,  and  Tran, V.C.
({\bf 2019}).  Stochastic Epidemic Models with Inference.
 {\it Lecture Notes in Mathematics, Mathematical Biosciences}, Vol. {\bf 2255}.
\bibitem{Barlow}
 Barlow, N. S., and  Weinstein, S. J. ({\bf 2020}). 
Accurate closed-form solution of the SIR epidemic model. {\it Physica D: Nonlinear Phenomena}, 
{\bf 408},  132540.  
\bibitem{Burton} Burton, T.A.  and Purnaras, I. K. 
({\bf 2017}).Global existence and uniqueness of solutions of
integral equations with delay: progressive contractions.  {\it Electronic Journal of Qualitative Theory of Differential Equations}, {\bf  49}, pp.~1--6.
\bibitem{DellAnna}
Dell’Anna, L. ({\bf 2020}). Solvable delay model for epidemic spreading: the case of Covid-19 in Italy. 
{\it Scientific  Reports},  {\bf 10}, 15763.

\bibitem{Fabricius} Fabricius, G., and Maltz, A. ({\bf 2020}). 
Exploring the threshold of epidemic spreading for a stochastic SIR model with local and global contacts. 
{\it Physica A: Statistical Mechanics and Applications}, {\bf 504}, 123208.
 \bibitem{Gairat} Gairat, A. ({\bf 1994}). Contact process without revival on tree. 
 In {\it Theoretical and applied aspects of mathematical researches: a collection of scientific works}. Preprint, Moscow  State University pp.~97--101. (In Russian).
\bibitem{Harko} Harko, T.,  Lobo, F.S.N. and Mak, M.K. ({\bf 2014}). 
Exact analytical solutions of the Susceptible-Infected-Recovered (SIR) epidemic model and of the SIR model with equal death and birth rates.
 {\it Applied Mathematics and Computation}, {\bf 236}, pp.~184--194.
\bibitem{Kendall} Kendall, D. (1956). Deterministic and stochastic epidemics 
in closed populations. {\it Proceedings of the Third Berkeley Symposium on Mathematical Statistics and Probability}, vol. IV, pp. ~149–165.
University of California Press, Berkeley and Los Angeles, Calif., 1956.
\bibitem{SIR-paper} Kermack, W.O. and    McKendrick, A.G.
({\bf 1927}). A contribution to the mathematical theory 
of epidemics.  {\it Proceedings of Royal Society A}, {\bf 115}, pp.~700--721.
\bibitem{Kroger} Kr\"{o}ger, M. and Schlickeiser ({\bf 2021}). Analytical solution of the SIR-model for the 
temporal evolution of epidemics. Part A: time-dependent reproduction factor.
{\it Journal of Physics A: Mathematical and Theoretical}. {\bf 53}, 505601. pp.~1--38.
\bibitem{Montagnon} Montagnon, P. ({\bf 2019}). 
A stochastic SIR model on a graph with epidemiological and
population dynamics occurring over the same time scale.
{\it Journal of Mathematical Biology}, {\bf 79}.
\bibitem{moreno} Moreno, Y., Pastor-Satorras, R. and  Vespignani, A. ({\bf 2002}).
Epidemic outbreaks in complex heterogeneous networks. {\it European Physical Journal B}, 
 {\bf 26}. 
\bibitem{Parker} Parker, A.E. ({\bf 2013}). Who solved the Bernoulli differential equation
and how did they do it? {\it The College Mathematics Journal}, {\bf 44:2}.
\bibitem{Schutz} 
Schutz, G.M., Brandaut, M.,  and Trimper, S. ({\bf 2008}).
Exact solution of a stochastic susceptible-infectious-recovered model.
{\it Physical Review E}, {\bf 78}, 061132.
\bibitem{Chen} Yuli Chen, Fawang Liu, Qiang Yu, Tianzeng Li ({\bf 2021}). 
 Review of fractional epidemic models. {\it Applied Mathematical Modelling},
 {\bf 97(4)}. 
\bibitem{zhang} Zhang, Z., Zhou, Z.,  Zou, T.  and Chen, G. ({\bf 2008}). 
 Fractal scale-free networks resistant to disease spread. {\it Journal of Statistical Mechanics: Theory and Experiment},
 {\bf 09}. 

\end{thebibliography}
\end{document}